\documentclass[12pt]{amsart}
\usepackage{amsfonts, amssymb, latexsym}
\usepackage{tikz}
\usepackage{wrapfig}

\newtheorem{theorem}{Theorem}[section]

\newtheorem{lemma}[theorem]{Lemma}

\newtheorem*{claim*}{Claim}

\newtheorem{Main Conjecture}[theorem]{Main Conjecture}

\theoremstyle{remark}

\theoremstyle{plain}

\newtheorem*{Thm}{Theorem}
\newtheorem*{ClaimA}{Claim~A}
\newtheorem*{ClaimB}{Claim~B}
\newtheorem*{ClaimC}{Claim~C}

\newcommand{\excise}[1]{}

\newcommand*\circled[1]{\tikz[baseline=(char.base)]{
  \node[shape=circle,draw,inner sep=1pt] (char) {$#1$};}}

\definecolor{darkgreen}{HTML}{008000}
\definecolor{nicered}{HTML}{d00000}
\definecolor{tangelo}{HTML}{e07000}

\begin{document}
\pagestyle{plain}

\mbox{}
\title{Newton polytopes and symmetric Grothendieck polynomials}
\author{Laura Escobar}
\author{Alexander Yong}
\address{Department of Mathematics \\ University of Illinois at Urbana--Champaign \\ Urbana, IL 61801 \\ USA}
\email{lescobar@illinois.edu, ayong@uiuc.edu}
\date{May 22, 2017}

\maketitle

\begin{abstract} \emph{Symmetric Grothendieck polynomials} are inhomogeneous versions of Schur polynomials that arise in combinatorial $K$-theory.
A polynomial has \emph{saturated Newton polytope} (SNP) if every lattice point in the polytope is an exponent vector.
We show Newton polytopes of these Grothendieck polynomials and their homogeneous components have SNP. Moreover, the Newton polytope of each homogeneous component is a permutahedron. This addresses recent conjectures of C.~Monical-N.~Tokcan-A.~Yong and of A.~Fink-K.~M\'esz\'aros-A.~St.~Dizier in this special case. 
\end{abstract}

\bigskip
Let $s_{\lambda}(x_1,\ldots,x_n)$ be the \emph{Schur polynomial},
which is the generating series for semistandard Young tableaux of shape $\lambda$ with entries in $[n]:=\{1,2,\ldots,n\}$. By work of C.~Lenart \cite[Theorem~2.2]{Lenart}, the {\bf symmetric Grothendieck polynomial}
is given by
\begin{equation}
\label{eqn:Glambda}
G_{\lambda}(x_1,\ldots,x_n)=\sum_{\mu} a_{\lambda\mu}s_{\mu}(x_1,\ldots,x_n).
\end{equation}
The sum is over partitions $\mu$ (identified with their Young diagrams in English notation) with $\leq n$ rows.
The $(-1)^{|\mu|-|\lambda|}a_{\lambda,\mu}$ counts the number of 
row and column strictly increasing
skew tableaux of shape $\mu/\lambda$ with entries in $[n]$ such that
the entries in row $r$ are weakly less than $r-1$.

By (\ref{eqn:Glambda}), $G_{\lambda}(x_1,\ldots,x_n)$ is an inhomogeneous
deformation of $s_{\lambda}(x_1,\ldots,x_n)$ and itself symmetric.
For example, if $n=3$ and $\lambda=(3,1,0)$,
\[G_{\lambda}(x_1,x_2,x_3)={\color{darkgreen}s_{(3,1)}}-(2{\color{blue}s_{(3,1,1)}}+{\color{blue}s_{(3,2,0)}})+2{\color{red}s_{(3,2,1)}}-{\color{tangelo}s_{(3,2,2)}}.\]
These polynomials appear in the study of $K$-theoretic Schubert calculus; we refer the reader to \cite{Lenart, Buch} and the references therein for additional discussion.

More generally, A.~Lascoux--M.-P.~Sch\"utzenberger \cite{LS} recursively defined (possibly nonsymmetric) Grothendieck polynomials associated to any 
permutation $w\in {\mathfrak S}_n$. We mention that A.~Buch \cite{Buch} discovered the 
\emph{set-valued tableaux} formula for $G_{\lambda}(x_1,\ldots,x_n)$; this formula is often taken as a definition in the literature. (Recently, C.~Monical \cite{Monical} found a bijection 
between the aforementioned rules of C.~Lenart and of
A.~Buch.)

The {\bf Newton polytope} of 
a polynomial 
$f=\sum_{\alpha\in {\mathbb Z}_{\geq 0}^n}c_{\alpha} x^{\alpha}\in {{\mathbb C}}[x_1,\ldots,x_n]$
is 
the convex hull of its exponent vectors, i.e., 
${\sf Newton}(f)={\sf conv}(\{\alpha:c_{\alpha}\neq 0\})\subseteq {\mathbb R}^n$.
In \cite{MTY},
$f$~is said to have {\bf saturated Newton polytope} (SNP) if $c_{\alpha}\neq 0$ whenever $\alpha\in {\sf Newton}(f)$. A study of SNP and algebraic combinatorics was given in \emph{loc.~cit.}

	If $\lambda=(\lambda_1\geq\lambda_2\geq\ldots\geq \lambda_n\geq 0)$, the {\bf permutahedron} ${\mathcal P}_{\lambda}\subseteq {\mathbb R}^n$ is the convex hull of the ${\mathfrak S}_n$-orbit of $\lambda$. This theorem extends the old fact that ${\sf Newton}(s_{\lambda}(x_1,\ldots,x_n))={\mathcal P}_{\lambda}$:

\begin{Thm}
$G_{\lambda}(x_1,\ldots,x_n)$ has SNP. In addition, each homogeneous component has SNP with Newton polytope being a
permutahedron (as specified below in (\ref{eqn:Newtonhomog})). 
\end{Thm}

\begin{wrapfigure}{r}{0.44\textwidth}
\begin{center}
\includegraphics[scale=0.25]{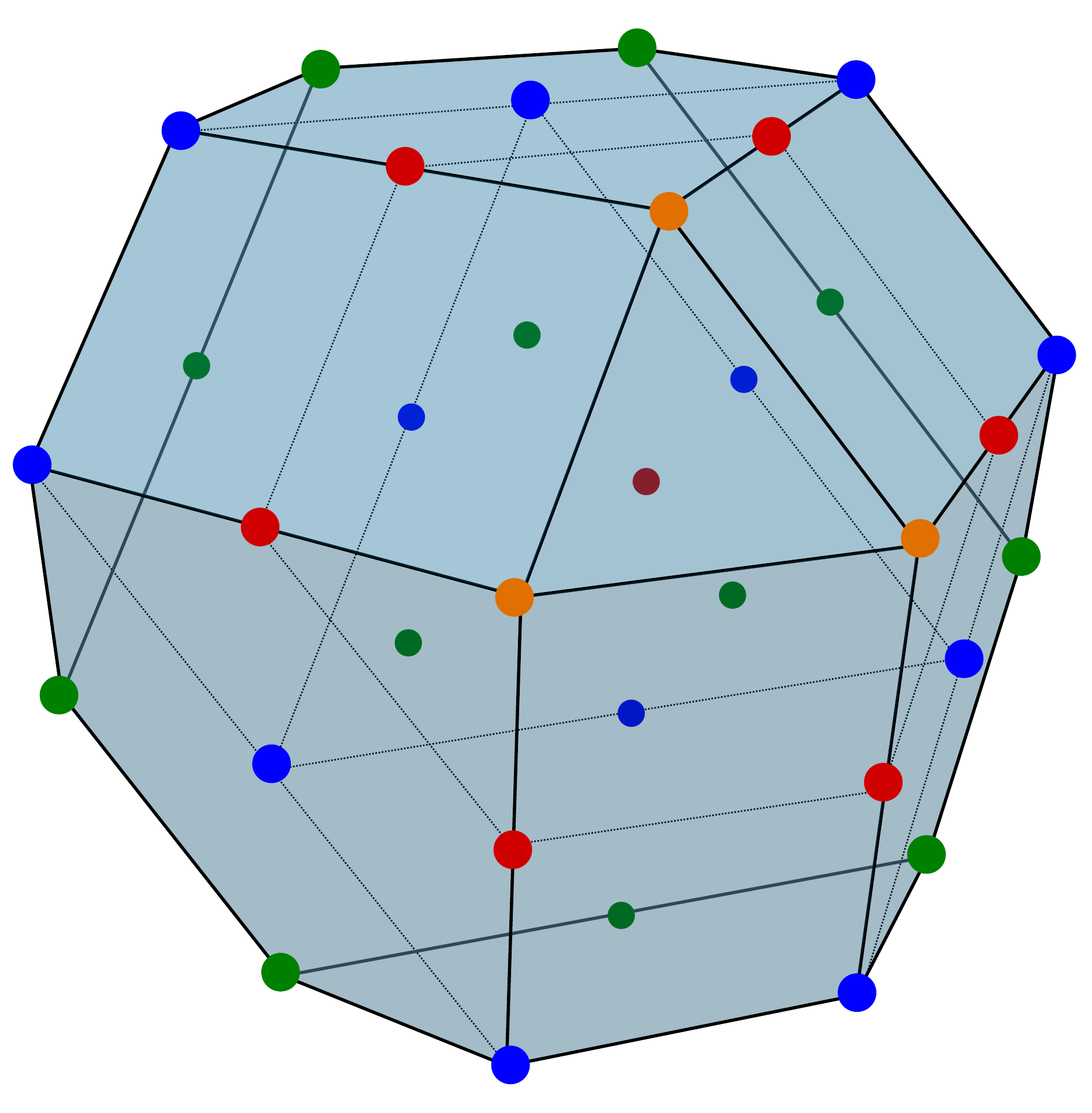}
\end{center}
\caption{The Newton polytope of $G_{\lambda}(x_1,x_2,x_3)$ for $\lambda=(3,1,0)$.  Each color indicates degree.} 
\end{wrapfigure}

The first assertion addresses \cite[Conjecture~5.5]{MTY} for the case that the permutation $w$ is {\bf Grassmannian} at position $n$;
that is $w(i)<w(i+1)$ unless $i=n$. 
The second assertion responds, in this case, to a conjecture
of A.~Fink-K.~M\'esz\'aros-A.~St.~Dizier \cite[Conjecture~5.1]{Meszaros}. In \emph{loc.~cit.}, these conjectures were proved
for the case that $w=1w'$ where $w'$ is a {\bf dominant permutation}, i.e., $w$ is $132$-avoiding.

\excise{
For a larger example, let $n=4$ and $\lambda=(4,2,1,0)$. Then
\begin{align*}
G_{\lambda}(x_1,x_2,x_3,x_4)= & s_{(4,2,1,0)}\\
& -(3s_{(4,2,1,1)}+s_{(4,2,2,0)}+2s_{(4,2,2,0)}+s_{(4,3,1,0)})\\
& +(6s_{(4,2,2,1)}+3s_{(4,3,1,1)}+2s_{(4,3,2,0)})\\
& -(5s_{(4,2,2,2)}+6s_{(4,3,2,1)}+s_{(4,3,3,0)})\\
& +(5s_{(4,3,2,2)}+3s_{(4,3,3,1)})\\
& -3s_{(4,3,3,2)}\\
& +s_{(4,3,3,3)}\\
\end{align*}
Notice that for each homogeneous component of degree $k$, there is a Schur function
$s_{\mu^{(k)}}$ where $\mu^{(k)}$ is the \emph{dominance order} maximum in the set $\{\mu:|\mu|=k,
a_{\lambda,\mu}\neq 0\}$. We explain below that this implies, by \cite[Proposition~2.5]{MTY}, that said component is SNP. Since the Newton polytope of $s_{\lambda}(x_1,\ldots,x_n)$
is ${\mathcal P}_{\lambda}$, the proof that $G_{\lambda}(x_1,x_2,x_3,x_4)$ itself is SNP is the statement that
\[{\mathcal P}_{(4,2,1,0)}\cup {\mathcal P}_{(4,3,1,0)}\cup{\mathcal P}_{(4,3,2,0)}\cup {\mathcal P}_{(4,3,3,0)}\cup {\mathcal P}_{(4,3,3,1)}\cup {\mathcal P}_{(4,3,3,2)}\cup {\mathcal P}_{(4,3,3,3)}\]
is convex. }

\medskip
\noindent
\emph{Proof of the Theorem:}  Let $\mu^{(0)}:=\lambda$. For $1\leq k\leq n$ define
$\mu^{(k)}$ to be $\mu^{(k-1)}$ with a  box added
in the northmost row $r$  such that $\mu^{(k-1)}_r-\mu^{(0)}_r<r-1$ and the addition of the box gives a Young diagram. 
Stop when no such $r$ exists or $k=n$.
Suppose we obtain $N$ such partitions.

Recall that {\bf dominance order} on partitions of a fixed size
is defined by 
$\theta\leq_D \delta$ if   $\sum_{j=1}^t \theta_j\leq
\sum_{j=1}^t \delta_j$ for $t\geq 1$.

\begin{ClaimA}
$\mu^{(k)}$ is the $\leq_D$-maximum among all shapes $\mu$ of size $|\lambda|+k$
such that $a_{\lambda,\mu}\neq 0$.
\end{ClaimA}
\noindent
\emph{Proof of Claim~A:}
The skew shape $\mu^{(k)}/\lambda$ consists of the $k$ boxes added to $\lambda$. We can inductively define a skew tableau $T_k$ of this shape by adding the minimum possible label to 
$T_{k-1}$ (in the box $\mu^{(k)}/\mu^{(k-1)}$)
that maintains row and column strictness. It is straightforward that this tableau exists and witnesses $a_{\lambda,\mu^{(k)}}\neq 0$.
\excise{
Suppose we add a label to row $r$. If there are no boxes to the left of the $k$-th box we can make this label be one more than the label on top, so the label is $\le r-2+1$.
Suppose there is a box to the left of the $k$-th box. Consider $T_k$ restricted to rows $r-1$ and $r$:
\begin{center}
\begin{tikzpicture}
\draw (4.5,1)--(4.5,0)--(0,0)--(0,.5)--(5.5,.5)--(5.5,1)--(2,1)--(2,0);
\draw (2.5,0)--(2.5,1);
\draw (3.5,0)--(3.5,1);
\draw (4,0)--(4,1);
\node at (4.25,.25) {$y_1$};
\node at (4.25,.75) {$x_1$};
\node at (3.75,.25) {$y_2$};
\node at (3.75,.75) {$x_2$};
\node at (3,.25) {$\cdots$};
\node at (3,.75) {$\cdots$};
\node at (2.25,.25) {$y_b$};
\node at (2.25,.75) {$x_b$};
\draw (1.5,0)--(1.5,.5);
\node at (1.75,.25) {$z_a$};
\node at (1,.25) {$\cdots$};
\node at (5,.75) {$\cdots$};
\draw (.5,0)--(.5,.5);
\node at (.25,.25) {$z_1$};
\end{tikzpicture}.
\end{center}
We know that $x_i\le r-(i+1)$ for $i\ge 1$, $y_i\le r-(i-1)$ for $i\ge 2$, and $a+b\le r-1$. Furthermore, since the $z_i$ have no labels on top, they must satisfy $z_i=i$.
If the smallest possible $y_1=r$ then $y_i=r-(i-1)$ for $i\ge2$. This forces $z_a=r-b$ contradicting $a+b\le r-1$.}

Let $\mu$ be a shape such that $a_{\lambda,\mu}\neq 0$.
Then $\mu_i\le \lambda_i+(i-1)$ for all $i$.
Suppose that $\mu\nleq_D\mu^{(k)}$
and let $r(>1)$ be the first row such that
$
\mu_{1}+\cdots+\mu_{r}>\mu^{(k)}_{1}+\cdots+\mu^{(k)}_{r}.
$
Then 
$$
\mu_{1}+\cdots+\mu_{r-1}\le\mu^{(k)}_{1}+\cdots+\mu^{(k)}_{r-1} \ \ \text{ and } \ \ \mu_{r}>\mu^{(k)}_{r}.
$$
This contradicts the construction of $\mu^{(k)}$ because by
$
\mu_r^{(k)}-\mu_r^{(0)}<\mu_r-\mu_r^{(0)}\le r-1
$
$\mu^{(k)}$ must have another box in row $r$.
\qed

R.~Rado's theorem \cite{rado} states that for two partitions
$\theta,\delta$ of the same size,
\begin{equation}
\label{eqn:rado}
{\mathcal P}_{\theta}\subseteq {\mathcal P}_{\delta}
\iff \theta\leq_D \delta.
\end{equation}
The Theorem's second assertion is immediate from
(\ref{eqn:rado}) and Claim~A. In fact if $G_{\lambda}[k]$
denotes the degree $k$ homogeneous component of $G_{\lambda}(x_1,\ldots,x_n)$ then 
\begin{equation}
\label{eqn:Newtonhomog}
{\sf Newton}(G_{\lambda}[k])={\mathcal P}_{\mu^{(k)}}.
\end{equation}

Let $v^{(k)}\in {\mathcal P}_{\mu^{(k)}}$. Thus $v^{(k)}$ is a nonnegative vector (being a convex combination of nonnegative vectors). 
Rado's theorem implies that $v^{(k)}$ is \emph{majorized} by $\mu^{(k)}$. That is, the rearrangement
$(v^{(k)})^\downarrow$ of the components of $v^{(k)}$ into decreasing order satisfies 
\begin{equation}
\label{eqn:letusmay19}
(v^{(k)})^\downarrow\leq_D \mu^{(k)}.
\end{equation}

Suppose
\[v=\sum_{k=0}^N c_k v^{(k)}
\text{ \ \ 
where \ \ $\sum_{k=0}^N c_k=1$ \ \  and \ \  $c_k\geq 0$}\] 
is a convex
combination of the vectors $v^{(k)}$. 

\begin{ClaimB}
$v$ is majorized by $\overline \mu:=\sum_{k=0}^N c_k\mu^{(k)}$. 
\end{ClaimB}
\noindent\emph{Proof of Claim B:} 
Let $v^\star:=\sum_{k=0}^N c_k (v^{(k)})^\downarrow$.
By (\ref{eqn:letusmay19}), for any $t\geq 1$ we have
$c_k\sum_{j=1}^t (v^{(k)})^\downarrow_j \leq c_k\sum_{j=1}^t \mu^{(k)}_j$.
By summing both sides over all
$k$ and interchanging the order of summation we conclude $v^\star$ is majorized by $\overline\mu$. It is a standard property of majorization that $a+b$ is majorized by $a^\downarrow+b^\downarrow$ \cite[Proposition~A.1.b]{majorization}. Thus $v$ is majorized by $v^\star$.
Now use that majorization (being a preorder) is transitive.  
\qed

\begin{ClaimC}
Suppose $|\overline{\mu}|-|\mu^{(0)}|=K$, then $\overline{\mu}$ is majorized by $\mu^{(K)}$.
\end{ClaimC}
\noindent
\emph{Proof of Claim~C:}
Let 
$r_k:=\text{row on which } k\text{-th box gets added to }\mu^{(k)}$ (so
$\mu^{(k)}=\mu^{(0)}+e_{r_1}+\cdots+e_{r_k}$,
where $e_i$ is a standard basis vector).

\begin{lemma}\label{lem:sumx} 
For any (row) $r$
\begin{align*}
\overline{\mu}_1+\cdots+\overline{\mu}_r=&\mu^{(0)}_1+\cdots+\mu^{(0)}_r+c_1+2c_2+\cdots+\ell c_{\ell}+\cdots+\ell c_N,
\end{align*}
where $\ell$ is the largest $i$ such that $r_i=r$.
\end{lemma}
\begin{proof} 
Suppose we added boxes $a,a+1,\ldots,a+b$ to row $r$ of $\mu^{(0)}$ in order to obtain $\mu^{(N)}$.

We write
$$\overline{\mu}_r= \underbrace{c_0\mu_r^{(0)}+\cdots+c_{a-1}\mu_r^{(a-1)}}_{\circled{1}}+ \underbrace{c_a\mu_r^{(a)}+\cdots+c_{a+b}\mu_r^{(a+b)}}_{\circled{2}}+ \underbrace{c_{a+b+1}\mu_r^{(a+b+1)}+\cdots+c_{N}\mu_r^{(N)}}_{\circled{3}}.
$$
Next, $\mu_r^{(0)}=\cdots=\mu_r^{(a-1)}$, so
$$
\circled{1} = (c_0+\cdots+c_{a-1})\mu_r^{(0)}.
$$
Since $\mu_r^{(a+i)}=\mu_r^{(0)}+(i+1)$ for $i=0,\ldots,b$, then
\begin{align*}
\circled{2} &=  (c_a+\cdots+c_{a+b})\mu_r^{(0)}+(c_a+2c_{a+1}+\cdots+(b+1)c_{a+b}).
\end{align*}
Finally, $\mu_r^{(a+b)}=\mu_r^{(a+b+1)}=\cdots=\mu_r^{(n)}$, so
\begin{align*}
\circled{3} &= (c_{a+b+1}+\cdots+c_{N})\mu_r^{(0)}+(c_{a+b+1}+\cdots+c_{N})(b+1).
\end{align*}
Therefore we conclude
$\overline{\mu}_r=\mu^{(0)}_r+c_a+2c_{a+1}+\cdots+
(b+1)c_{a+b}+\cdots+(b+1)c_N$.
The lemma then follows by a simple induction.
\end{proof}

The following is immediate from the definitions.
\begin{lemma}\label{lem:sumla} Let $b_r=\mu^{(N)}_r-\mu^{(0)}_r$, i.e. $b_r$ is the number of extra boxes $\mu^{(N)}$ has in row $r$. For any $r<r_k$
$$
\mu^{(k)}_1+\cdots+\mu^{(k)}_r=\mu^{(0)}_1+\cdots+\mu^{(0)}_r+(b_{1}+\cdots+b_{r}).
$$
\end{lemma}

Let $\ell$ be the largest $i$ such that $r_i=r$. We consider two cases.

\noindent
{\sf Case 1} ($r<r_K$): Observe that
\[c_1+2c_2+\cdots+\ell c_{\ell}+\cdots+\ell c_N =(c_1+\cdots+c_N)+(c_2+\cdots+c_N)+\cdots+(c_{\ell}+\cdots+c_N)\\
\le \ell.
\]
Since $b_1+\cdots+b_r$ equals the number of boxes placed from rows $1$ through $r$ and the $\ell$-th box is the last box placed in row $r$, then 
$\ell=b_1+\cdots+b_r$. Combining this equality with the inequality just derived, we see
\begin{equation}
\label{eqn:justderived}
c_1+2c_2+\cdots+\ell c_{\ell}+\cdots+\ell c_N\leq b_{1}+\cdots+b_{r}.
\end{equation}
By (\ref{eqn:justderived}) together with Lemmas \ref{lem:sumx} and \ref{lem:sumla},
$${\overline \mu}_1+\cdots+{\overline \mu}_r\le \mu_1^{(K)}+\cdots+\mu_r^{(K)}.
$$

\noindent
{\sf Case 2} ($r\ge r_K$): Here, we notice that 
\begin{equation}
\label{eqn:acc}
\mu^{(K)}_1+\cdots+\mu^{(K)}_r=\mu^{(0)}_1+\cdots+\mu^{(0)}_r+K.
\end{equation}
Observe that
\begin{equation}
\label{eqn:vvc}
 c_1+2c_2+\cdots+\ell c_{\ell}+\cdots+\ell c_N \le c_1+2c_2+\cdots+Nc_N=K,
\end{equation}
where the equality follows from Lemma \ref{lem:sumx}.
Apply Lemma \ref{lem:sumx} to the left hand side of (\ref{eqn:vvc}) and use (\ref{eqn:acc}) to replace $K$ on the right hand side, to conclude 
${\overline \mu}_1+\cdots+{\overline \mu}_r\le \mu_1^{(K)}+\cdots+\mu_r^{(K)}.
$
Hence in either case, ${\overline \mu}\leq_D \mu^{(K)}$, as required.
\qed

Let $w_1,\ldots,w_M$ be any exponent vectors of
$G_{\lambda}(x_1,\ldots,x_n)$. SNPness means that
if $w\in {\sf conv}\{w_1,\ldots,w_M\}$ is a lattice point then $[x^w]G_{\lambda}(x_1,\ldots,x_n))\neq 0$.
Suppose that
$|w|=K$. 
Without loss of generality, $M=N$ and there is a unique vector $w_k$ with $|w_k|=k$. Then by Claim~B,
$w$ is majorized by ${\overline \mu}$. Claim~C says ${\overline \mu}$ is majorized by $\mu^{(K)}$ and hence $w^\downarrow\leq_D \mu^{(K)}$.
By (\ref{eqn:rado}) we conclude 
$w\in {\mathcal P}_{\mu^{(K)}}$, which by (\ref{eqn:Newtonhomog})
completes the proof of the Theorem. Indeed, we have shown that ${\sf Newton}(G_{\lambda}(x_1,\ldots,x_n))=\bigcup_{k=0}^N {\mathcal P}_{\mu^{(k)}}$.
\qed

\section*{Acknowledgements}
We thank Karola M\'esz\'aros and Avery St.~Dizier for a helpful communication. We used SAGE code of C.~Monical during our investigation.
AY was supported by an NSF grant.

\end{document}